\documentclass{amsart}
\usepackage{amssymb}
\usepackage{amsfonts}

\newtheorem{theorem}{Theorem}
\theoremstyle{plain}

\newtheorem{corollary}{Corollary}

\newtheorem{definition}{Definition}
\newtheorem{example}{Example}

\newtheorem{lemma}{Lemma}

\newtheorem{proposition}{Proposition}

\numberwithin{equation}{section}

\newcommand\Aut{\mathrm{Aut}}
\newcommand\Inn{\mathrm{Inn}}
\newcommand\Sym{\mathrm{Sym}}

\newcommand\LSec{\mathrm{LSec}}
\newcommand\RSec{\mathrm{RSec}}

\newcommand\al{\alpha}
\newcommand\bt{\beta}
\newcommand\gm{\gamma}

\newcommand\la{\langle}
\newcommand\ra{\rangle}

\newcommand\st{${}^\ast$ }

\begin{document}
\title{AG-groups and other classes of right Bol quasigroups}
\author{M. Shah}
\address{DEPARTMENT OF MATHEMATICS, QUAID-E-AZAM UNIVERSITY, ISLAMABAD,
PAKISTAN.}
\email{shahmaths\_problem@hotmail.com}
\author{S. Shpectorov}
\email{sergeys@for.mat.bham.ac.uk}
\address{SCHOOL OF MATHEMATICS, UNIVERSITY OF BIRMINGHAM, UK.}
\email{s.shpectorov@bham.ac.uk}
\author{A. Ali}
\email{dr\_asif\_ali@hotmail.com}
\address{DEPARTMENT OF MATHEMATICS, QUAID-E-AZAM UNIVERSITY, ISLAMABAD,
PAKISTAN.}

\begin{abstract}
By a result of Sharma, right Bol quasigroups are obtainable from right Bol
loops via an involutive automorphism. We prove that the class of AG-groups,
introduced by Kamran, is obtained via the same construction from abelian
groups. We further introduce a new class of Bol\st quasigroups, which
turns out to correspond, as above, to the class of groups.

Sharma's correspondence allows an efficient implementation and we present
some enumeration results for the above three classes.
\end{abstract}

\maketitle

\section{Introduction}

By definition, an AG-group (also called
an LA-group)\ $G$ is a set with a binary operation satisfying the left
invertive law : $(xy)z=(zy)x$ for all $x,y,z\in G$ and also having left
identity and left inverses. From these axioms it follows that the left
inverse also is the right inverse and thus it becomes a two sided inverse.
In particular, AG-groups belong to the class of quasigroups \cite{SA3}.
AG-groups were introduced in the PhD thesis of Kamran \cite{MPT} and they
first appeared in print in \cite{MK}. Some of the basic properties of
AG-groups were discussed in \cite{SA1}. For example, associativity,
commutativity, and the existence of identity are equivalent properties for
AG-groups. In particular, among the groups only the abelian groups are
AG-groups. For a geometric interpretation of AG-groups see \cite{SA6}.
AG-groupoids (also called LA-semigroups), which generalize AG-groups, have
applications in flock theory, see \cite{MN}. For additional sources on
AG-groupoids, we suggest \cite{MY2}, and also \cite{SSA}. It was noticed in
\cite{SA3} that AG-groups belong to the class of right Bol quasigroups. It
is well known that right Bol quasigroups and right Bol loops have
applications in differential geometry \cite{SA}. In \cite{SA1} enumeration
of AG-groups was proposed as an interesting problem. In \cite{SA2} the
enumeration was carried out computationally up to order $12$. In this paper
we completely classify AG-groups by showing that every AG-group arises from
an abelian group via an involutive automorphism.

\begin{theorem} \label{main}
Suppose $G$ is an abelian group and $\al\in\Aut(G)$ satisfying $\al^2=1$.
Define a new binary operation on $G$ by $a\cdot b:=\al(a)+b$. Then $G_\al=(G,\cdot)$
is an AG-group. Furthermore, every AG-group is obtainable in this way.
Finally, the AG-groups $G_\al$ and $H_\beta$ are isomorphic if and only if
the abelian groups $G$ and $H$ are isomorphic and automorphisms $\al$ and
$\beta$ are conjugate.
\end{theorem}

This description of the class of AG-groups allows us to classify various
subclasses of them. For example, it easily follows from Theorem \ref{main}
that the AG-group $G_\al$ is a group if and only if $\al$ is the
identity automorphism of the abelian group $G$. In the similar spirit
let an AG-group be called {\em involutory} if its every element is an involution,
that is, it satisfies $a^2=e$, where $e$ is the (left) identity element. The
following is a corollary of Theorem \ref{main}.

\begin{theorem} \label{involutory}
An AG-group $G_\al$ is involutory if and only if $\al$ is the minus
identity automorphism that is $\al(g)=-g$ for all $g\in G$. In particular,
there is a natural bijection between abelian groups and involutory
AG-groups.
\end{theorem}

The groups of order one and two are the only cyclic groups for which
the identity automorphism is the same as minus identity. In particular,
for all orders $n>2$ there exists a non-associative AG-group.

There have been a lot of publications (see for example,  \cite{DKM}) about
the multiplication groups of loops and quasigroups. By definition, the
multiplication group $M(Q)$ of a quasigroup $Q$ is the subgroup of $\Sym(Q)$
generated by all left and right translations. The multiplication group of an
AG-group was studied in \cite{SA4} where it was established that for a
nonassociative AG-group of order $n$ its multiplication group is nonabelian
of order $2n$ and, correspondingly, the so called inner mapping group has order
two. Based on Theorem \ref{main}, we can give a more precise description
of the multiplication group.

\begin{theorem} \label{multiplication group}
Suppose $G_\al$ is a non-associative AG-group, that is, $\al$ is
non-identity. Then $M(G_\alpha)$ is isomorphic to the semidirect product
$G:\la\al\ra$. Note also that the order two group $\la\al\ra$ is the inner
mapping group $I(G_\al)$, that is, the stabilizer in $M(G_\al)$ of the
identity element.
\end{theorem}

The construction of the AG-groups from the abelian groups, as described in
Theorem \ref{main}, can easily be implemented in a computer algebra system.
In fact, we implemented it in GAP \cite{NV} and were able to enumerate all
AG-groups up to the order 2009. We stopped at the number because we used
the small group library of GAP as our source of abelian groups. The method
can easily be extended to much greater orders, as long as the abelian groups
of that order are available.  As a sample of the computation, we provide
here (see Table \ref{table1}) the information about the number of AG-groups
up to order 20.
\begin{center}
{\tiny
\begin{table}[h] \label{table1}
\begin{tabular}{|l|l|l|l|l|l|l|l|l|l|l|l|l|l|l|l|l|l|l|l|l|}
\hline
Order & $3$ & $4$ & $5$ & $6$ & $7$ & $8$ & $9$ & $10$ & $11$ & $12$ & $13$
& $14$ & $15$ & $16$ & $17$ & $18$ & $19$ & $20$ \\ \hline
Group & $1$ & $2$ & $1$ & $1$ & $1$ & $3$ & $2$ & $1$ & $1$ & $2$ & $1$
& $1$ & $1$ & $5$ & $1$ & $2$ & $1$ & $2$ \\ \hline
Other & $1$ & $2$ & $1$ & $1$ & $1$ & $7$ & $3$ & $1$ & $1$ & $6$ & $1$
& $1$ & $3$ & $24$ & $1$ & $3$ & $1$ & $6$ \\ \hline
Total & $2$ & $4$ & $2$ & $2$ & $2$ & $10$ & $4$ & $2$ & $2$ & $8$ & $2$
& $2$ & $4$ & $29$ & $2$ & $5$ & $2$ & $8$ \\ \hline
\end{tabular}
\vskip .5cm
\caption{Number of AG-groups of order $n$, $3\le n\le 20$}
\end{table}
}
\end{center}

The correspondence between the classes of abelian groups and AG-groups is
very simple, so naturally, we were wondering whether a similar construction
had been known. And indeed, we found a paper by Sharma \cite{SHA} establishing
a correspondence between the classes of left Bol loops and left Bol
quasigroups. By duality, there is a similar correspondence between right
Bol loops and right Bol quasigroups. This dual correspondence is essentially
the same as our correspondence. Clearly, the class of abelian
groups is a subclass of the class of right Bol loops. It is not so
immediately clear, but still can be shown that the class of AG-groups
is a subclass of the class of right Bol quasigroups. Hence our
correspondence is simply a special case of Sharma's correspondence
adjusted for the case of right Bol loops. In this sense, our Theorem
\ref{main} shows that the class of AG-groups is the counterpart of
the class of abelian groups under Sharma's correspondence. We consider it an
interesting problem to determine which classes of quasigroups are the
counterparts of other subclasses of right Bol loops, such as say,
the class of groups or the class of Moufang loops. In this paper
we give an answer to the first of these questions, namely, we provide
the axioms for the class of quasigroups corresponding to the class of groups.

\begin{definition}
A {\em right Bol\st quasigroup} is a quasigroup satisfying
$$a(bc\cdot d)=(ab\cdot c)d$$
for all elements $a,b,c,d$.
\end{definition}

Note that the substitution $d=b$ turns the above equality into the
right Bol law, which shows that the class of the right {\em Bol${}^\ast$}
quasigroups is a subclass of right Bol quasigroups. In the future we
will just speak of Bol quasigroups and Bol\st quasigroups, skipping
`right'.

\begin{theorem} \label{main2}
Suppose $G$ is a group and $\al\in\Aut(G)$ satisfying $\al^2=1$.
Define a new binary operation on $G$ by $a*b:=\al(a)b$. Then $G_\al=(G,*)$
is a Bol\st quasigroup. Furthermore, every Bol\st quasigroup is
obtainable in this way. Finally, the Bol\st quasigroups $G_\al$ and
$H_\beta$ are isomorphic if and only if the groups $G$ and $H$ are
isomorphic and automorphisms $\al$ and $\beta$ are conjugate.
\end{theorem}

\section{Preliminaries}

The following property of AG-groups was established in \cite{SA3}.

\begin{lemma} \label{L1}
Every AG-group satisfies the identity $(ab\cdot c)d=a(bc\cdot d)$. In other
words, every AG-group is a Bol\st quasigroup.
\end{lemma}

We now embark on proving Theorem \ref{main}. We start with the first claim
in that theorem.

\begin{proposition} \label{AG from abelian and alpha}
Let $G$ be an abelian group under addition and let $\al\in Aut(G)$ be such
that $\al^2=1$. Define $x\cdot y=\al(x)+y$ for all $x,y\in G$.
Then $G_\al=(G,\cdot)$ is an AG-group with left identity $e=0$.
\end{proposition}

\begin{proof}
We start by checking the left invertive law in $G_\al$. Let $x,y,z\in G$.
Then $xy\cdot z=\al(\al(x)+y)+z=\al^{2}(x)+\al(y)+z=x+\al(y)+z$, since
$\al^2=1$. Similarly, $zy\cdot x=z+\al(y)+x$, and so $zy\cdot x=z+\al(y)+x=
x+\al(y)+z=xy\cdot z$.

It is easy to see that $0$ is the left identity in $G_\al$. Indeed,
$0x=\al(0)+x=0+x$, for all $x\in G$. Finally, we claim that $\al(-x)$ is
the left inverse of $x$. Indeed, $\al(-x)x=\al(\al(-x))+x=-x+x=0$.

This shows that $G_\al$ is an AG-group.
\end{proof}

We next need to show that every AG-group can be obtained as above. Let $G$
be an AG-group with a left identity $e$. We first show how to build
an abelian group from $G$.

\begin{proposition} \label{abelian from AG}
Consider the set $G$ together with the new operation $+$ defined as follows:
$$x+y:=xe\cdot y,$$
for all $x,y\in G$. Then $(G,+)$ is an abelian group. The zero element of this
group is $e$ and, for every $x\in G$, the inverse $-x$ is equal to $x^{-1}e$.
\end{proposition}

\begin{proof}
We start by checking associativity of addition. Let $x,y,z\in G$. Then $(x+y)+z=
(xe\cdot y)e\cdot z$. Using Lemma \ref{L1} with $a=xe$, $b=y$, $c=e$,
and $d=z$, we get that $(xe\cdot y)e\cdot z=xe\cdot(ye\cdot z)=x+(y+z)$.
Therefore, $(x+y)+z=x+(y+z)$, proving associativity.

Commutativity of addition follows essentially by the definition. Indeed,
$x+y=xe\cdot y=ye\cdot x=y+x$ by the left invertive law. Similarly,
$e+x=ee\cdot x=ex=x$. Now by commutativity $e$ is the zero element of
$(G,+)$. Finally, $x^{-1}e+x=(x^{-1}e)e\cdot x=((ee)x^{-1})x=ex^{-1}\cdot x=
x^{-1}x=e$. Again, commutativity shows that $x^{-1}e$ is the inverse $-x$.
\end{proof}

We remark that in place of the identity $e$ we could use any fixed element
$c\in G$. Namely, if we define addition via: $x\oplus y:=xc\cdot y$ then we
again get an abelian group, whose zero element is $c$ and where the inverses
are computed as follows: $\ominus x:=x^{-1}c$. The proof is essentially the
same. Furthermore, the groups obtained for different elements $c$ are all
isomorphic. Namely, the isomorphism between $(G,+)$ and $(G,\oplus)$ is
given by $x\mapsto x*c$.

Our next step is to prove that the mapping $\al:G\to G$ defined by
$g\mapsto ge$ is an involutive automorphism of the abelian group $(G,+)$.

\begin{proposition} \label{alpha from AG}
For all $x,y\in G$, we have $\al(x+y)=\al(x)+\al(y)$ and, furthermore,
$\al^2=1$. Therefore, $\al$ is an involutive automorphism of $(G,+)$.
\end{proposition}

\begin{proof}
We first note that by the left invertive law $\al^2(x)=xe\cdot e=ee\cdot x=
ex=x$ for all $x\in G$. Therefore, $\al^2=1$, the identity mapping on $G$.
By Lemma \ref{L1}, $\al(x+y)=\al(xe\cdot y)=(xe\cdot y)e=x(ey\cdot e)=
x(ye)$. On the other hand, $\al(x)+\al(y)=(xe)e\cdot ye$. We saw above that
$xe\cdot e=x$, hence $\al(x)+\al(y)=x(ye)$, which we have shown to be equal
to $\al(x+y)$. Therefore, $\al$ is an automorphism.
\end{proof}

The last two results show that every AG-group $G$ canonically defines an abelian
group $(G,+)$ and its involutive automorphism $\al$. It remains to see that the
AG-group $G$ can be recovered from $(G,+)$ and $\al$ as in Proposition
\ref{AG from abelian and alpha}.

\begin{proposition} \label{reverse}
Suppose $G$ is an AG-group and let $(G,+)$ and $\al$ be the corresponding
abelian group and its involutive automorphism. Then for all $x,y\in G$
we have $xy=\al(x)+y$. That is, $G=G_\al$.
\end{proposition}

\begin{proof}
This is clear: indeed, $\al(x)+y=(xe\cdot e)y=xy$. We used
the identity $xe\cdot e=x$, which we showed before.
\end{proof}

We now turn to homomorphisms between AG-groups.

\begin{proposition} \label{homomorphism}
Suppose $G$ and $H$ are abelian groups and let $\al\in Aut(G)$ with $\al^{2}=1$
and $\bt\in Aut(H)$ with $\bt^{2}=1$. Then the set of homomorphisms between
AG-groups $G_\al$ and $H_\bt$ coincides with the set of group homomorphisms
$\pi:G\to H$ satisfying $\pi\al=\bt\pi$.
\end{proposition}

\begin{proof}
Suppose $\pi:G_\al\to H_\bt$ is a homomorphism of AG-groups, that is, it is a
mapping $G\to H$ such that $\pi(gh)=\pi(g)\pi(h)$. By cancellativity,
$\pi$ sends the left identity of $G_\al$ to the left identity of $H_\bt$.
Therefore, for $x,y\in G$, we have $\pi(x+y)=\pi(xe\cdot y)=\pi(xe)\pi(y)=
\pi(x)\pi(e)\cdot\pi(y)=\pi(x)e\cdot\pi(y)=\pi(x)+\pi(y)$. This shows that
$\pi$ is a homomorphism of abelian groups. Next, let $x\in G$. Then
$\pi\al(x)=\pi(xe)=\pi(x)e=\bt\pi(x)$. Since $x\in G$ is arbitrary, we conclude
that $\pi\al=\bt\pi$.

For the converse, suppose that $\pi:G\to H$ is a homomorphism of abelian
groups and that $\pi$ satisfies $\pi\al=\bt\pi$. Then, for $x,y\in G$, we have
$\pi(xy)=\pi(\al(x)+y)=\pi\al(x)+\pi(y)=\bt\pi(x)+\pi(y)=\pi(x)\pi(y)$. Hence
$\pi$ is a homomorphism of AG-groups.
\end{proof}

This allows to complete the proof of Theorem \ref{main}.

\begin{corollary} \label{isomorphism}
Two AG-groups $G_\al$ and $H_\bt$ are isomorphic if and only if there is
an isomorphism $\pi$ between $G$ and $H$, satisfying $\pi\al\pi^{-1}=\bt$.
\end{corollary}

\begin{proof}
Immediately follows from Proposition \ref{homomorphism}. Indeed, if $\pi$ is
bijective then the condition $\pi\al=\bt\pi$ is equivalent to
$\pi\al\pi^{-1}=\bt$.
\end{proof}

We record here a further corollary of Proposition \ref{homomorphism}, which
describes the full automorphism group of the AG-group $G_\al$.

\begin{corollary} \label{automorphisms}
The automorphism group of the AG-group $G_\al$ coincides with $C_{Aut(G)}(\al)$,
the centralizer in $Aut(G)$ of the involution $\al$.
\end{corollary}

\begin{proof}
If $G_\al=H_\bt$ (and so $G=H$ and $\al=\bt$) then the condition $\pi\al=\bt\pi=
\al\pi$ means simply that $\pi\in Aut(G)$ must commute with $\al$.
\end{proof}

It is interesting that the involutory twist construction can be used
repeatedly.

\begin{proposition} \label{second time}
Let $(G,\circ )$ be an AG-group with a left identity $e$. Let $\al\in\Aut(G)$
such that $\al^2=1$. Define $x\cdot y=\al(x)\circ y$ for all $x,y\in G$. Then
$(G,\cdot)$ is again an AG-group.
\end{proposition}

\begin{proof}
Initially this had an independent proof. However, with all the theory that we
have developed, this result follows easily. Indeed, by Theorem \ref{main}, the
AG-group $(G,\circ)$ must be equal to $G_\bt$, for an abelian group $G$ and
its involutory automorphism $\bt$.

Note that this means that $x\circ y=\bt(x)+y$, where, as usual, plus indicates
addition in the abelian group $G$. According to Corollary \ref{automorphisms},
$\al$ is an automorphism of the group $G$ commuting with $\bt$. In particular,
$\gm=\bt\al$ is again an involutory automorphism of $G$.

We now notice that $x\cdot y=\al(x)\circ y=\bt(\al(x))+y=\gm(x)+y$. This
means that $(G,\cdot)$ is simply the AG-group $G_\gm$.
\end{proof}

\section{Particular classes of AG-groups}

It is natural to ask when the AG-group $G_\al$ is associative, that is,
a group. It was shown in \cite{SA1} that for AG-groups associativity is equivalent
to commutativity and also to the property that the left identity $e$ is a
two-sided identity. We can show that, in fact, $G_\al$ is never a group,
when $\al\ne 1$.

\begin{proposition} \label{group case}
Suppose $G$ is an abelian group and $\al\in\Aut(G)$ with $\al^2=1$. Then $G_\al$
is a group if and only if $\al=1$.
\end{proposition}

\begin{proof}
If $\al=1$ then $\al(x)+y=x+y$, hence $G_\al$ is simply the group $G$. Conversely,
assume that $G_\al$ is a group. Note that $e=0$ is the left identity of $G_\al$,
since $0\cdot x=\al(0)+x=0+x=x$. However, in a group the left identity is the same
as the right identity. Therefore, for all $x\in G$, we must have $x\cdot 0=x$.
However, $x\cdot 0=\al(x)+0=\al(x)$. Hence, $\al(x)=x$ for all $x\in G$, which
means that $\al=1$.
\end{proof}

This proof already verifies that $G_\al$ is a group whenever it has a two-sided
identity. Quite similarly, if $G_\al$ is commutative then for every $x\in G$ we
have $x\cdot 0=0\cdot x=x$. On the other hand, $x\cdot 0=\al(x)+0=\al(x)$. Hence
we must have that $\al(x)=x$ for all $x\in G$, and so $G_\al=G$ is a group. This
shows that indeed commutativity is also equivalent to associativity.

The second interesting class of AG-groups is the class of involutary AG-groups.
Recall from the introduction that an AG-group $G$ is called involutory if its
every nontrivial element is an involution, i.e., $x^2=e$ for all $x\in G$
where $e$ is the left identity of $G$.

\begin{proposition} \label{involutary}
Suppose $G$ is an abelian group and $\al\in\Aut(G)$ with $\al^2=1$. Then $G_\al$
is an involutory if and only if $\al =-1$. (This means that $\al(x)=-x$ for all
$x\in G$.)
\end{proposition}

\begin{proof}
Recall that $x\cdot x=\al(x)+x$, so $x\cdot x=e=0$ if and only if $\al(x)+x=0$,
that is, $\al(x)=-x$, so $G_{\alpha }$ is involutory if and only if $\al(x)=-x$
for all $x$.
\end{proof}

As a consequence, we get the following.

\begin{corollary} \label{always}
For every order $n\ge 3$ there exists a non-associaitive AG-group of order $n$.
\end{corollary}

\begin{proof}
Indeed, we can take $G=C_n$, the cyclic group of order $n$, and $\al=-1$. When
$n\ge 3$, we have $\al\ne 1$, which means that $G_\al$ is non-associative by
Proposition \ref{group case}.
\end{proof}

Since for every prime order $n=p>2$ there exists exactly one abelian group, the
cylic group $C_p$, and since $\Aut(C_p)\cong C_{p-1}$, which has a unique element
of order two, we have the following result.

\begin{corollary} \label{odd prime}
For every prime order $n=p\ge 3$, there is only one non-associaitive AG-group
of order $n$.
\end{corollary}

\section{Some examples}

For illustration we provide some examples. The case of the prime order has been
dealt with in the preceding section.

\begin{example}
For order $6$, we have only one abelian group, namely $C_6$. Since $\Aut(C_6)$
has only one nontrivial involution, there are exactly two AG-groups of order
6, one associative, $C_6$, and one non-associative, namely, $(C_6)_\al$, where
$\al=-1$.
\end{example}

The same is true for all orders $2p$, where $p$ is an odd prime. So this case
is similar to the case of the odd prime order.

\begin{example}
For order $12$, there are exactly two abelian groups, namely $C_{12}$ and
$C_6\times C_2$. In the first case, $\Aut(C_{12})$ is an elementary
abelian group of order four. Hence its every element can be used to construct
a new AG-group. This gives us four AG-groups (one associative, one
non-associative involutory, and two further non-associative non-involutory).
In the second case, $\Aut(C_6\times C_2)$ is isomorphic to $C_2\times Sym(3)$,
and so is nonabelian of order $12$. In addition to the identity element, this
group has three conjugacy classes of involutions. Hence, in this case, too, we
get four different AG-groups.

In total, we obtain eight AG-groups of order $12$, out of which six are
non-associative.
\end{example}

\begin{example}
Let us consider the order $2009=7^2\cdot 41$. Again, there are two abelian
groups of this order, $C_{2009}$ and $C_{287}\times C_7$. In the first case
the automorphism group is abelian, containing three involutions. Hence this
group leads to four AG-groups. The automorphism group of $C_{287}\times C_7$
is isomorphic to $C_{40}\times GL(2,7)$. This group has five conjugacy classes
on involutions in addition to the identity element, hence in this case we
obtain six different AG-groups.

In total, there are 10 different AG-groups of order 2009, out of which eight
are non-associative.
\end{example}

\section{A GAP package for computing with AG-groups}

V. Sorge  and the first author developed a GAP package AGGROUPOIDS which, in
particular, contains functions dealing with AG-groups. They are based on the theory
developed in this paper.

There are four main functions:
\begin{itemize}
\item The function {\tt NrAllSmallNonassociativeAGGroups(n)} returns the total
number of nonassociative AG-groups of order $n$ provided that the SmallGroups
library contains the list of groups of order $n$. This restriction will
be lifted in the future, since all abelian groups of a given order are easy
to construct.
\item The function {\tt AllSmallNonassociativeAGGroups(n)} returns the list
of all non-associative AG-groups of the given order. Each AG-group is
represented as a GAP quasigroup.
\item The function {\tt NrAllSmallNonassociativeAGGroupsFromAnAbelianGroup(G)}
returns the total number of non-associative AG-groups that can be obtained from
the abelian group $G$. This is equal to the number of conjugacy classes
of involutions in $\Aut(G)$.
\item  The function {\tt AllSmallNonassociativeAGGroupsFromAnAbelianGroup(G)}
returns the list of non-associative AG-groups obtainable from $G$, again
as GAP quasigroups.
\end{itemize}

The entire package (not limited to these four functions) will shortly be
available from the GAP repository.

\section{Multiplication group of an AG-group}

The concept of the multiplication group of a loop and, more generally, a quasigroup
is well known. In a quasigroup $Q$, multiplication on the left (or right)
by an element $x\in Q$ is a permutation $L_x$ (respectively, $R_x$) of $Q$ called
the {\em left} (respectively, {\em right}) {\em translation} by $x$.  The set
of all left translations is called the {\em left section}, and similarly, the
set of right translations is called the {\em right section} of $Q$. We will
write $\LSec$ and $\RSec$ for the left and right sections, respectively.
Therefore, $\LSec=\{L_x\mid x\in Q\}$ and $\RSec=\{R_x\mid x\in Q\}$.

The multiplication group $M(Q)$ is the subgroup of the symmetric group $Sym(Q)$
generated by $\LSec\cup\RSec$. If $Q$ is a loop, the stabilizer in $M(Q)$ of the
identity is called the {\em inner mapping group} and denoted $\Inn(Q)$.

Since every AG-group $G_\al$ is a quasigroup we can consider its multiplication
group $M(G_\al)$. Since $G_\al$ has a left identity $0$, we can generalize
the concept of the inner mapping group to the class of AG-groups by setting
$\Inn(G_\al)$ to be the stabilizer of $0$ in $M(G_\al)$.

\begin{proposition} \label{multiplication group facts}
Let $G$ be an abelian group and $\al\in\Aut(G)$ with $\al^2=1$. Then the following
hold:
\begin{enumerate}
\item[{\rm (1)}] $M(G_\al)=\LSec\cup\RSec$;
\item[{\rm (2)}] $\Inn(G_\al)=\la\al\ra$;
\item[{\rm (3)}] $\LSec$ is a normal subgroup of $M(G_\al)$ and it is naturally
isomorphic to $G$;
\item[{\rm (4)}] $\RSec=\al\LSec$; and
\item[{\rm (5)}] $M(G_\al)$ is isomorphic to the semidirect product of $G$ with
the cyclic group $\la\al\ra$.
\end{enumerate}
\end{proposition}

\begin{proof}
First note that the mapping $\psi:x\mapsto L_x$ is a homomorphism from $G$ to
$Sym(G)$. Indeed, $L_{x+y}(z)=(x+y)\cdot z=\al(x+y)+z=\al(x)+\al(y)+z=
x\cdot(\al(y)+z)=x\cdot(y\cdot z)=L_x(L_y(z))$ for all $z\in G$. This means that
$L_{x+y}$ is indeed the product of $L_x$ and $L_y$. Since $\psi$ is a homomorphism,
its image $\LSec$ is a subgroup of $\Sym(G)$. Furthermore, if $L_x(z)=z$ for
some $z\in G$ then $\al(x)+z=z$, which implies that $x=0$. Therefore, $\psi$ is
injective and so it is an isomorphism from $G$ onto $\LSec$.

Next, note that $\al(z)=\al(z)+0=z\cdot 0=R_0(z)$. This means that $\al=R_0$ is
an element of $RSec$. Furthermore, $R_x(z)=\al(z)+x=\al(z)+\al^2(x)=
\al(\al(x))+\al(z)=\al(\al(x)+z)=(\al L_x)(z)$. This means that
$R_x=\al L_x$ for all $x\in G$, that is, $RSec$ is the coset of $LSec$ containing
$\al$.

We now turn to part (1). We claim that $\al$ normalizes $LSec$. Indeed,
$(\al L_x\al)(z)=\al L_x(\al(z))=\al(\al(x)+\al(z))=x+z=\al(\al(x))+z=L_{\al(x)}(z)$.
Thus, $\al L_x\al=L_{\al(x)}$, proving that $\al$ normalizes the subgroup
$LSec$. Since $RSec=\al LSec$, we conclude that every element of $RSec$ normalizes
$LSec$, which means that $LSec$ is normal in $M(G_\al)$. Also, it means that
$M(G_\al)=\la LSec,\al\ra$, which implies that $LSec$ has index at most two
in $M(G_{\al})$. (This proves (1).) To be more precise, the index is two if and only
if $\al\not\in LSec$. Clearly, $\al$ fixes $0$ and, as we have already seen, the
only element of $LSec$ fixing $0$ is $L_0$, the identity element of $LSec$. Hence
$LSec$ has index two in $M(G_\al)$ if and only if $\al\ne 1$.

From the above, we also have that $|\Inn(G_\al)|=|\al|$, since $\LSec$ is regular
on $G$ and so $|\Inn(G_\al)|$ is equal to the index of $\LSec$ in $M(G_\al)$.
Since $\al$ fixes $0$, we have $\al\in\Inn(G_\al)$, which implies (2). Parts (3)
and (4) have already been proven. Finally, since $\al\not\in\LSec$ and
$M(G_\al)=\la\LSec,\al\ra$, (5) follows as well.
\end{proof}

As an example of how the multiplication group can be used to identify the
AG-group, we present the following result.

\begin{theorem} \label{AT11}
Suppose $M=M(G_\al)$ for a non-associative AG-group $G_\al$ and $M\cong D_{2n}$. Then
either $G$ is the Klein four-group (and so $n=4$) or $G\cong C_{n}$ is cyclic. In
the latter case $\al=-1$, and hence $G_\al$ is involutory.
\end{theorem}

\begin{proof}
First of all, since $G_\al$ is non-associative, $\al$ is a nontrivial automorphism
of $G$ and so $n=|G|\ge 3$. By Theorem \ref{multiplication group facts}, the abelian group
$G$ is isomorphic to an index two subgroup of $M$. From this, it immediately follows
that either $n=4$ and $G$ is the Klein four-group, or $n\ge 3$ is arbitrary and $G$
is cyclic. Finally, in the cyclic case, since $M(G_\al)$ is isomorphic to the
semidirect product of $G$ and $\la\al\ra$, we conclude that $\al$ inverts every element
of $G$ and so $\al=-1$.
\end{proof}

We also give a general characterization of all groups that arise as
multiplication group of an AG-group.

\begin{theorem} \label{general}
A nonabelian group $M$ is isomorphic to a multiplication group of some non-associative
AG-group if and only if $M\cong T\rtimes R$ where $T$ is abelian and $|R|=2$.
\end{theorem}

\begin{proof}
If $M=M(G_\al)$ then $M=G\rtimes\la\al\ra$ and so all the claimed properties hold.
Conversely, suppose $M=T\rtimes R$ where $T$ is abelian and $|R|=2$. Let $\al\in\Aut(T)$
be the automorphism induced by the generator of $R$ on $T$. Then $M\cong M(T_\al)$
by Proposition \ref{multiplication group facts} (5).
\end{proof}

\section{Sharma's correspondence}

In his paper \cite{SA1} from 1976 Sharma proved the following theorem. We recall
that the identity
$(ab\cdot c)b=a(bc\cdot b)$
is known as the right Bol identity. The loops (respectively, quasigroups)
satisfying this identity are called the right Bol loops (respectively,
right Bol quasigroups).

\begin{theorem} \label{Sharma}
Suppose $G$ is a right Bol loop and $\al\in\Aut(G)$ satisfying $\al^2=1$.
Define a new binary operation on $G$ by $a*b:=\al(a)b$. Then $G_\al=(G,*)$
is a right Bol quasigroup. Furthermore, every right Bol quasigroup is
obtainable in this way. Finally, the right Bol quasigroups $G_\al$ and
$H_\beta$ are isomorphic if and only if the right Bol loops $G$ and $H$ are
isomorphic and automorphisms $\al$ and $\beta$ are conjugate.
\end{theorem}

In reality Sharma proved the ``left'' version of this theorem, but we
switched to the above, ``right'' version because it matches better our own
results.

In particular, Sharma's theorem implies that every right Bol quasigroup
automatically has a left identity element.

We note that Sharma's construction is essentially the same as ours,
except it is done for a different, larger class of objects, the Bol loops
instead of abelian groups. In other words, what we proved in Theorem \ref{main}
means simply that the class of AG-groups is the counterpart of the subclass of
abelian groups under Sharma's correspondence. It would be interesting to ask
what are the counterparts of other subclasses of Bol loops, such as, say,
groups or Moufang loops. We leave the Moufang loops case as an open question,
but we have an answer for the class of groups.

Recall from the introduction that by a {\em Bol\st quasigroup} we mean a
quasigroup satisfying
$$a(bc\cdot d)=(ab\cdot c)d$$
for all $a,b,c,d$. Note that this is clearly a subclass of Bol
quasigroups. In particular, every Bol\st quasigroup automatically has
a left identity element.

\begin{theorem} \label{generalized}
Suppose $G$ is a group and $\al\in\Aut(G)$ satisfying $\al^2=1$.
Define a new binary operation on $G$ by $a*b:=\al(a)b$. Then $G_\al=(G,*)$
is a Bol\st quasigroup. Furthermore, every Bol\st quasigroup is
obtainable in this way. Finally, the Bol\st quasigroups $G_\al$ and
$H_\beta$ are isomorphic if and only if the groups $G$ and $H$ are
isomorphic and automorphisms $\al$ and $\beta$ are conjugate.
\end{theorem}

\begin{proof}
Let us first see that $G_\al$ as above satisfies the identity
$$a*((b*c)*d)=((a*b)*c)*d.$$
Indeed, $a*((b*c)*d)=\al(a)(\al(\al(b)c)d)=\al(a)\al^2(b)\al(c)d=
\al(a)b\al(c)d$, since $\al^2=1$. Similarly, $((a*b)*c)*d=\al(\al(\al(a)b)c)d=
\al^3(a)\al^2(b)\al(c)d=\al(a)b\al(c)d$. So the identity holds, proving that
$G_\al$ is a Bol\st quasigroup.

Conversely, assume that $(G,*)$ is a Bol\st quasigroup with left identity
$e$. For $x\in G$, define $\al(x)=x*e$ and also, for $x,y\in G$, define
$xy=\al(x)y$. We need to see that (1) $G$ with this new product is a group;
(2) $\al$ is an automorphism of this group of order two; and (3)
$(G,*)=G_\al$.

First of all, for $x,y,z\in G$, $x(yz)=\al(x)*(\al(y)*z)=(x*e)*((y*e)*z)$.
By the identity, the latter is equal to $(((x*e)*y)*e)*z$ and this is equal
to $(x*((e*y)*e)*z$. On the other hand, $(xy)z=\al(\al(x)*y)*z=
(((x*e)*y)*e)*z$, so we have $x(yz)=(xy)z$ for all $x,y,z\in G$, proving that
the new operation is associative. Cancelativity is clear, so we have an
associative quasigroup, hence a group. Note that $e$ is the identity
element of the group, since $ex=\al(e)*x=(e*e)*x=e*x=x$.

For (2), we first need to show that $\al$ is a permutation of order two:
$\al^2(x)*z=((x*e)*e)*z=x*((e*e)*z)=x*z$, and so by cancellativity,
$\al^2(x)=x$. Thus, $\al^2=1$. To show that $\al$ is an automorphism,
we compute:
$\al(xy)=(xy)*e=(\al(x)*y)*e=((x*e)*y)*e=x*((e*y)*e)=x*(y*e)$
and $\al(x)\al(y)=\al(\al(x))*\al(y)=x*(y*e)$. Thus,
$\al(xy)=\al(x)\al(y)$. Finally, (3) is clear since $x*y=\al^2(x)*y=
\al(x)y$. Hence $x*y=\al(x)y$, which means that $(G,*)=G_\al$.

For the final claim in the theorem, we note that the proofs of Proposition
\ref{homomorphism} and Corollary \ref{automorphisms} depend neither on
commutativity of the group operation, nor on the left invertive identity,
so they fully apply in our present case.
\end{proof}

The package AGGROUPOIDS mentioned above also contains functions
enumerating Bol\st quasigroups and Bol quasigroups based on Theorem
\ref{generalized} and Sharma's Theorem \ref{Sharma}. In Tables 2 and 3
we provide the counting for the Bol\st quasigroups and Bol
quasigroups up to order 20 and 30, respectively.
\begin{center}
{\tiny
\begin{table}[h] \label{table2}
\begin{tabular}{|l|l|l|l|l|l|l|l|l|l|l|l|l|l|l|l|l|l|l|l|l|}
\hline
Order & $3$ & $4$ & $5$ & $6$ & $7$ & $8$ & $9$ & $10$ & $11$ & $12$ & $13$
& $14$ & $15$ & $16$ & $17$ & $18$ & $19$ & $20$ \\ \hline
Group & $1$ & $2$ & $1$ & $2$ & $1$ & $5$ & $2$ & $2$ & $1$ & $5$ & $1$
& $2$ & $1$ & $14$ & $1$ & $5$ & $1$ & $5$ \\ \hline
Non-group & $1$ & $2$ & $1$ & $2$ & $1$ & $12$ & $3$ & $2$ & $1$ & $14$ & $1$
& $2$ & $3$ & $88$ & $1$ & $9$ & $1$ & $13$ \\ \hline
Total & $2$ & $4$ & $2$ & $4$ & $2$ & $17$ & $5$ & $4$ & $2$ & $19$ & $2$
& $4$ & $4$ & $102$ & $2$ & $14$ & $2$ & $18$ \\ \hline
\end{tabular}
\bigskip
\caption{Number of Bol\st quasigroups of order $n$, $3\le n\le 20$}
\end{table}
}
\end{center}
\begin{center}
{\tiny
\begin{table}[h] \label{table3}
\begin{tabular}{|l|l|l|l|l|l|l|l|l|l|l|l|l|l|l|}
\hline
Order & $3$ & $4$ & $5$ & $6$ & $7$ & $8$ & $9$ & $10$ & $11$ & $12$ & $13$
& $14$ & $15$ & $16$ \\ \hline
Bol loop & $1$ & $2$ & $1$ & $1$ & $1$ & $3$ & $1$ & $1$ & $1$ & $2$ & $1$
& $1$ & $1$ & $5$ \\ \hline
Other & $1$ & $2$ & $1$ & $1$ & $1$ & $7$ & $3$ & $1$ & $1$ & $6$ & $1$
& $1$ & $3$ & $24$ \\ \hline
Total & $2$ & $4$ & $2$ & $4$ & $2$ & $41$ & $5$ & $4$ & $2$ & $23$ & $2$
& $4$ & $10$ & $16581$ \\ \hline
\hline
Order & $17$ & $18$ & $19$ & $20$ & $21$ & $22$ & $23$ & $24$ & $25$ & $26$ & $27$
& $28$ & $29$ & $30$ \\ \hline
Bol loop & $1$ & $2$ & $1$ & $1$ & $1$ & $3$ & $1$ & $1$ & $1$ & $2$ & $1$
& $1$ & $1$ & $5$ \\ \hline
Other & $1$ & $2$ & $1$ & $1$ & $1$ & $7$ & $3$ & $1$ & $1$ & $6$ & $1$
& $1$ & $3$ & $24$ \\ \hline
Total & $2$ & $16$ & $2$ & $30$ & $12$ & $4$ & $2$ & $\ge 713$ & $5$ & $4$ & $36$
& $28$ & $2$ & $\ge 22$ \\ \hline

\end{tabular}
\bigskip
\caption{Number of Bol quasigroups of order $n$, $3\le n\le 30$}
\end{table}
}
\end{center}
It might be worth mentioning that we can enumerate Bol\st quasigroups
for much larger orders, as long as the list of groups of that order is
available. For Bol quasigroups, we can only go up to the order 31,
as the list of Bol loops of order 32 is an open problem.

\end{document}